\documentclass[10pt]{amsart}

\newcounter{minutes}
\divide\time by 60
\newcounter{hours}
\multiply\time by 60
\addtocounter{minutes}{-\time}

\usepackage{amsfonts}
\usepackage{ifthen}
\usepackage{amsthm}
\usepackage{amsmath}
\usepackage{graphicx}
\usepackage{amscd,amssymb}

\newcommand{\doo}{\operatorname{d\!}}

\title{Functional inequalities for modified Bessel functions} 

\author{\'Arp\'ad Baricz}
\author{Saminathan Ponnusamy}
\author{Matti Vuorinen}
\address{Department of Economics, Babe\c{s}-Bolyai University,
Cluj-Napoca 400591, Romania} \email{bariczocsi@yahoo.com}
\address{Department of Mathematics, Indian Institute of Technology Madras, Chennai 600036,
India} \email{samy@iitm.ac.in}
\address{Department of Mathematics, University of Turku, Turku 20014,
Finland} \email{vuorinen@utu.fi}

\newtheorem{theorem}{Theorem}
\newtheorem{open}{Question}

\newcounter{alphabet}
\newcounter{tmp}
\newenvironment{Theorem}[1][]{\refstepcounter{alphabet}%
\bigskip%
\noindent%
{\bf Theorem \Alph{alphabet}}%
\ifthenelse{\equal{#1}{}}{}{ (#1)}%
{\bf .} \itshape}{\vskip 8pt}

\newenvironment{Lemma}[1][]{\refstepcounter{alphabet}%
\bigskip%
\noindent%
{\bf Lemma \Alph{alphabet}}%
\ifthenelse{\equal{#1}{}}{}{ (#1)}%
{\bf .} \itshape}{\vskip 8pt}
\keywords{Functional inequalities; Modified Bessel functions;
Convexity with respect to H\"older means; Log-convexity; Geometrical
convexity; Gamma-gamma distribution; Tur\'an-type inequality. \newline  \small \texttt{File:~\jobname .tex,
          printed: \number\year-0\number\month-0\number\day,
          \thehours.\ifnum\theminutes<10{0}\fi\theminutes} }

\begin{document}

\maketitle


\begin{abstract}
In this paper our aim is to show some mean value
inequalities for the modified Bessel functions of the first and
second kinds. Our proofs are based on some bounds for the
logarithmic derivatives of these functions, which are in fact
equivalent to the corresponding Tur\'an type inequalities for these
functions. As an application of the results concerning the modified
Bessel function of the second kind we prove that the cumulative
distribution function of the gamma-gamma distribution is
log-concave. At the end of this paper several open problems are
posed, which may be of interest for further research.
\end{abstract}



\section{\bf Introduction}

Let us consider the probability density function
$\varphi:\mathbb{R}\rightarrow(0,\infty)$ and the reliability (or
survival) function $\overline{\Phi}:\mathbb{R}\rightarrow(0,1)$ of
the standard normal distribution, defined by
$$\varphi(u)=\frac{1}{\sqrt{2\pi}}e^{-u^2/2} \ \ \ \ \ \mbox{and} \ \ \ \ \
\overline{\Phi}(u)=\frac{1}{\sqrt{2\pi}}\int_{u}^{\infty}e^{-t^2/2}\doo
t.$$ The function $r:\mathbb{R}\rightarrow(0,\infty),$ defined by
$$r(u)=\frac{\overline{\Phi}(u)}{\varphi(u)}=e^{u^2/2}\int_u^{\infty}e^{-t^2/2}\doo t,$$
is known in literature as Mills' ratio \cite[sect. 2.26]{mitri} of
the standard normal distribution, while its reciprocal $1/r,$
defined by $1/r(u)=\varphi(u)/\overline{\Phi}(u),$ is the so-called
failure (hazard) rate, which arises frequently in economics and
engineering sciences. Recently, among other things, Baricz
\cite[Corollary 2.6]{bariczmills} by using the Pinelis' version of
the monotone form of l'Hospital's rule (see
\cite{pinelis,andersonH,avv1} for further details) proved the
following result concerning the Mills ratio of the standard normal
distribution:

\begin{Theorem}\label{theoA}
If $u_1,u_2>u_0,$ where $u_0\approx1.161527889\dots$ is the
unique positive root of the transcendent equation
$u(u^2+2)\overline{\Phi}(u)=(u^2+1)\varphi(u),$ then the following
chain of inequalities holds
\begin{equation}\label{funci}\frac{2r(u_1)r(u_2)}{r(u_1)+r(u_2)}\leq
r\left(\frac{u_1+u_2}{2}\right)\leq\sqrt{r(u_1)r(u_2)}\leq
r(\sqrt{u_1u_2}) \leq\frac{r(u_1)+r(u_2)}{2}\leq
r\left(\frac{2u_1u_2}{u_1+u_2}\right).\end{equation} Moreover, the
first, second, third and fifth inequalities hold for all $u_1,u_2$
positive real numbers, while the fourth inequality is reversed if
$u_1,u_2\in(0,u_0).$ In each of the above inequalities equality
holds if and only if $u_1=u_2.$
\end{Theorem}

We note here that, since Mills' ratio $r$ is continuous, the second
and third inequalities in \eqref{funci} mean actually that under the
aforementioned assumptions Mills' ratio is log-convex and
geometrically concave on the corresponding interval. More precisely,
by definition a function $f:[a,b]\subseteq\mathbb{R}\to(0,\infty)$
is log-convex if $\ln f$ is convex, i.e. if for all
$u_1,u_2\in[a,b]$ and $\lambda\in[0,1]$ we have
$$f(\lambda u_1+(1-\lambda)u_2)\leq \left[f(u_1)\right]^{\lambda}\left[f(u_2)\right]^{1-\lambda}.$$
Similarly, a function $g:[a,b]\subseteq(0,\infty)\to(0,\infty)$ is
said to be geometrically (or multiplicatively) convex if $g$ is
convex with respect to the geometric mean, i.e. if for all
$u_1,u_2\in[a,b]$ and $\lambda\in[0,1]$ we have
$$g\left(u_1^\lambda u_2^{1-\lambda}\right)\leq \left[g(u_1)\right]^{\lambda}\left[g(u_2)\right]^{1-\lambda}.$$
We note that if $f$ and $g$ are differentiable then $f$ is
log-convex if and only if $u\mapsto f'(u)/f(u)$ is increasing on
$[a,b]$, while $g$ is geometrically convex if and only if $u\mapsto
ug'(u)/g(u)$ is increasing on $[a,b].$ A similar definition and
characterization of differentiable log-concave and geometrically
concave functions also holds.

Mean value inequalities similar to those presented above appear also
in the recent literature explicitly or implicitly for other special
functions, like the Euler gamma function and its logarithmic
derivative (see for example the paper \cite{alzer} and the
references therein), the Gaussian and Kummer hypergeometric
functions, generalized Bessel functions of the first kind, general
power series (see the papers \cite{anderson,bariczj1,bariczj2}, and
the references therein), Bessel and modified Bessel functions of the
first kind (see \cite{bariczexpo,neuman,eneuman}).

In this paper, motivated by the above results, we are mainly
interested in mean value functional inequalities concerning modified
Bessel functions of the first and second kinds. The detailed content
is as follows: in section 2 we present some preliminary results
concerning some tight lower and upper bounds for the logarithmic
derivative of the modified Bessel functions of the first and second
kinds. These results will be applied in the sequel to obtain some
interesting chain of inequalities for modified Bessel functions of
the first and second kinds analogous to \eqref{funci}. To achieve
our goal in section 2 we present some monotonicity properties of
some functions which involve the modified Bessel functions of the
first and second kinds. Section 3 is devoted to the study of the
convexity with respect to H\"older (or power) means of modified
Bessel functions of the first and second kinds. The results stated
here complete and extend the results from section 2. As an
application of our results stated in section 2, in section 4 we show
that the cumulative distribution function of the three parameter
gamma-gamma distribution is log-concave for arbitrary shape
parameters. This result may be useful in problems of information
theory and communications. Finally, in section 5 we present some
interesting open problems, which may be of interest for further
research.



\section{\bf Monotonicity properties of some functions involving modified Bessel functions}

As usual, in what follows let us denote by $I_{\nu}$ and $K_{\nu}$
the modified Bessel functions of the first and second kinds of real
order $\nu$ (see \cite{watson}), which are in fact the linearly
independent particular solutions of the second order modified Bessel
homogeneous linear differential equation \cite[p. 77]{watson}
\begin{equation}\label{diffeq}
u^2v''(u)+uv'(u)-(u^2+\nu^2)v(u)=0.
\end{equation}
Recall that the modified Bessel function $I_{\nu}$ of the first kind
has the series representation \cite[p. 77]{watson}
$$I_{\nu}(u)=\sum_{n\geq0}\frac{(u/2)^{2n+\nu}}{n!\Gamma(n+\nu+1)},$$
where $\nu\neq-1,-2,\dots$ and $u\in\mathbb{R},$ while the modified
Bessel function of the second kind $K_{\nu}$ (called sometimes as
the MacDonald or Hankel function), is usually defined also as
\cite[p. 78]{watson}
$$K_{\nu}(u)=\frac{\pi}{2}\frac{I_{-\nu}(u)-I_{\nu}(u)}{\sin\nu\pi},$$
where the right-hand side of this equation is replaced by its
limiting value if $\nu$ is an integer or zero. We note that for all
$\nu$ natural and $u\in\mathbb{R}$ we have $I_{\nu}(u)=I_{-\nu}(u),$
and from the above series representation $I_{\nu}(u)>0$ for all
$\nu>-1$ and $u>0.$ Similarly, by using the familiar integral
representation \cite[p. 181]{watson}
\begin{equation}\label{integr}K_{\nu}(u)=\int_0^{\infty}e^{-u\cosh
t}\cosh(\nu t)\doo t,\end{equation} which holds for each $u>0$ and
$\nu\in\mathbb{R},$ one can see easily that $K_{\nu}(u)>0$ for all
$u>0$ and $\nu\in\mathbb{R}.$

The following results provide some tight lower and upper bounds for
the logarithmic derivatives of the modified Bessel functions of the
first and second kinds $I_{\nu}$ and $K_{\nu}$ and will be used
frequently in the sequel.

\begin{Lemma}\label{lem1}
For all $u>0$ and $\nu>0$ the following inequalities hold
\begin{equation}\label{eq4}
\sqrt{\frac{\nu}{\nu+1}u^2+\nu^2}<\frac{uI_{\nu}'(u)}{I_{\nu}(u)}<\sqrt{u^2+\nu^2}.
\end{equation}
Moreover, the right-hand side of \eqref{eq4} holds true for all
$\nu>-1.$
\end{Lemma}

\begin{Lemma}\label{lem2}
For all $u>0$ and $\nu>1$ the following inequalities hold
\begin{equation}\label{eq1}
-\sqrt{\frac{\nu}{\nu-1}u^2+\nu^2}<\frac{uK_{\nu}'(u)}{K_{\nu}(u)}<-\sqrt{u^2+\nu^2}.
\end{equation}
Moreover, the right-hand side of \eqref{eq1} holds true for all
$\nu\in\mathbb{R}.$
\end{Lemma}

The left-hand side of \eqref{eq4} was proved for $u>0$ and positive
integer $\nu$ by Phillips and Malin \cite{phillips}, and later by
Baricz \cite{baricz1} for $u>0$ and $\nu>0$ real. The right-hand
side of \eqref{eq4} appeared first in Gronwall's paper
\cite{gronwall} for $u>0$ and $\nu>0$ (motivated by a problem in
wave mechanics), it was proved also by Phillips and Malin
\cite{phillips} for $u>0$ and $\nu\geq 1$ integer, and recently by
Baricz \cite{baricz1} for $u>0$ and $\nu\geq -1/2$ real (motivated
by a problem in biophysics; see \cite{penfold}). For this inequality
the case $u>0$ and $\nu>-1$ real has been proved recently in
\cite{baricz2}.

The left-hand side of \eqref{eq1} was proved first by Phillips and
Malin \cite{phillips} for $u>0$ and $\nu>1$ positive integer, and
was extended to the case $u>0$ and $\nu>1$ real recently by Baricz
\cite{baricz2}. Finally, the right-hand side of \eqref{eq1} was
proved first by Phillips and Malin \cite{phillips} for $u>0$ and
$\nu\geq 1$ integer, and later extended to the case of $u>0$ and
$\nu$ real arbitrary by Baricz \cite{baricz1}.

It is worth mentioning that the inequalities \eqref{eq4} and
\eqref{eq1}, which have been proved recently also by Segura
\cite{segura}, are in fact equivalent to the Tur\'an type
inequalities for the modified Bessel functions of the first and
second kinds. For further details the interested reader is referred
to \cite{baricz1,baricz2,bapo,lana,segura} and to the references
therein.

Our first main result reads as follows.

\begin{theorem}\label{th1}
The following assertions are true:
\begin{enumerate}
\item[\bf (a)] $u\mapsto uI_{\nu}'(u)/I_{\nu}^2(u)$ is
strictly decreasing on $(0,\infty)$ for all $\nu\geq 1;$
\item[\bf (b)] $u\mapsto uI_{\nu}'(u)/I_{\nu}(u)$ is
strictly increasing on $(0,\infty)$ for all $\nu>-1;$
\item[\bf (c)] $u\mapsto \sqrt{u}I_{\nu}(u)$ is strictly
log-concave on $(0,\infty)$ for all $\nu\geq 1/2;$
\item[\bf (d)] $u\mapsto u^2I_{\nu}'(u)/I_{\nu}^2(u)$ is
strictly decreasing on $(0,\infty)$ for all $\nu\geq \nu_0,$ where
$\nu_0\approx1.373318506\dots$ is the positive root of the cubic
equation $8\nu^3-9\nu^2-2\nu-1=0.$
\end{enumerate}
In particular, for all $u_1,u_2>0$ and $\nu\geq \nu_0$ the following
chain of inequalities holds
\begin{equation}\label{chain1}\frac{2I_{\nu}(u_1)I_{\nu}(u_2)}{I_{\nu}(u_1)+I_{\nu}(u_2)}\leq
I_{\nu}\left(\frac{2u_1u_2}{u_1+u_2}\right)\leq
I_{\nu}\left(\sqrt{u_1u_2}\right)\leq
\sqrt{I_{\nu}(u_1)I_{\nu}(u_2)} \leq
\sqrt{\frac{u_1+u_2}{2\sqrt{u_1u_2}}}\cdot
I_{\nu}\left(\frac{u_1+u_2}{2}\right).\end{equation} Moreover, the
second and third inequalities hold true for all $\nu>-1,$ and the
fourth inequality holds true for all $\nu\geq 1/2.$ In each of the
above inequalities equality hold if and only if $u_1=u_2.$
\end{theorem}

We recall that part {\bf (b)} of Theorem \ref{th1} was proved for
$\nu>0$ by Gronwall \cite{gronwall}. Notice also that recently
Baricz \cite{baricz2} in order to prove the right-hand side of
\eqref{eq4} proved implicitly part {\bf (b)} of Theorem \ref{th1}.
For reader's convenience we recall below that proof. Moreover, we
give a somewhat different proof of this part, and two other
completely different proofs.

We note that part {\bf (c)} of Theorem \ref{th1} improves the result
of Sun and Baricz \cite{sun}, who proved that the function $u\mapsto
uI_{\nu}(u)$ is log-concave on $(0,\infty)$ for all $\nu\geq 1/2.$
Recently, Baricz and Neuman \cite{neuman} conjectured that the
modified Bessel function $I_{\nu}$ of the first kind is strictly
log-concave on $(0,\infty)$ for all $\nu>0.$ As far as we know, this
conjecture is still open and the much sharper result of this kind is
of part {\bf (c)} of Theorem \ref{th1}.

\begin{proof}[\bf Proof of Theorem \ref{th1}]

First we prove the monotonicity and log-concavity properties stated
above.

{\bf (a)} Recall that the modified Bessel function of the first kind
$I_{\nu}$ is a particular solution of the second-order differential
equation \eqref{diffeq} and thus
\begin{equation}\label{eq3}I_{\nu}''(u)=(1+\nu^2/u^2)I_{\nu}(u)-(1/u)I_{\nu}'(u).\end{equation}
Using \eqref{eq3} and the left-hand side of \eqref{eq4}, we obtain
that for all $u>0$ and $\nu\geq 1$
\begin{align*}
\frac{\doo}{\doo
u}\left[\frac{uI_{\nu}'(u)}{I_{\nu}^2(u)}\right]=\left[\frac{1}{uI_{\nu}(u)}\right]
\left[u^2+\nu^2-2\left[\frac{uI_{\nu}'(u)}{I_{\nu}(u)}\right]^2\right]<\left[\frac{1}{uI_{\nu}(u)}\right]\left[{-\nu^2+\frac{1-\nu}{1+\nu}u^2}\right]\leq0.
\end{align*}

{\bf (b)} Consider the Tur\'anian
$$\Delta_{\nu}(u)=I_{\nu}^2(u)-I_{\nu-1}(u)I_{\nu+1}(u),$$ which in
view of the recurrence relations
$$I_{\nu-1}(u)=(\nu/u)I_{\nu}(u)+I_{\nu}'(u)$$ and
$$I_{\nu+1}(u)=-(\nu/u)I_{\nu}(u)+I_{\nu}'(u),$$ can be rewritten as
follows
$$\Delta_{\nu}(u)=(1+\nu^2/u^2)I_{\nu}^2(u)-[I_{\nu}'(u)]^2.$$ Using
\eqref{eq3} we get
$$\Delta_{\nu}(u)=\frac{1}{u}I_{\nu}^2(u)\left[\frac{uI_{\nu}'(u)}{I_{\nu}(u)}\right]'.$$
It is known (see \cite{tiru,baricz2}) that the Tur\'an-type
inequality $\Delta_{\nu}(u)>0$ holds for all $u>0$ and $\nu>-1,$ and
hence the required result follows. We may note incidentally that the
result of this part actually follows also from the right-hand side
of \eqref{eq4}. More precisely, it is easy to see that the function
$u\mapsto uI_{\nu}'(u)/I_{\nu}(u)$ satisfies the differential
equation $uv'(u)=u^2+\nu^2-v^2(u),$ and using the right-hand side of
\eqref{eq4} it is clearly strictly increasing on $(0,\infty)$ for
all $\nu>-1.$ It is important to add here that in fact the
right-hand side of \eqref{eq4} and the Tur\'an-type inequality
$\Delta_{\nu}(u)>0$ are equivalent (see \cite{baricz1,baricz2}).

A third proof of this part can be obtained as follows. By using the
infinite series representation of the modified Bessel function of
the first kind we just need to show that the function
$$u\mapsto \frac{uI_{\nu}'(u)}{I_{\nu}(u)}=\left.\sum_{n\geq0}\frac{(2n+\nu)(u/2)^{2n}}{n!\Gamma(\nu+n+1)}\right/\sum_{n\geq0}\frac{(u/2)^{2n}}{n!\Gamma(\nu+n+1)}$$
is strictly increasing on $(0,\infty)$ for all $\nu>-1.$ To do this
let us recall the following well-known result (see
\cite{biernacki,ponnusamy}): {\em Let us consider the power series
$f(u)=a_0+a_1u+{\dots}+a_nu^n+{\dots}$ and
$g(u)=b_0+b_1u+{\dots}+b_nu^n+{\dots},$ where for all $n\geq 0$
integer $a_n\in\mathbb{R}$ and $b_n>0,$ and suppose that both
converge on $(0,\infty).$ If the sequence $\{a_n/b_n\}_{n\geq 0}$ is
strictly increasing, then the function $u\mapsto f(u)/g(u)$ is
strictly increasing too on $(0,\infty).$} We note that we can see
easily that the above result remains true in the case of even
functions. Thus, to prove that $u\mapsto uI_{\nu}'(u)/I_{\nu}(u)$ is
indeed strictly increasing it is enough to show that the sequence
$\{\alpha_n\}_{n\geq0},$ defined by $\alpha_n=2n+\nu$ for all
$n\geq0,$ is strictly increasing, which is certainly true.

Finally, a fourth proof is as follows. By using the Weierstrassian
factorization
$$I_{\nu}(u)=\frac{u^{\nu}}{2^{\nu}\Gamma(\nu+1)}\prod_{n\geq 1}\left(1+\frac{u^2}{j_{\nu,n}^2}\right),$$
where $\nu>-1$ and $j_{\nu,n}$ is the $n$th positive zero of the
Bessel function $J_{\nu}$ of the first kind, we obtain that
$$\frac{\doo}{\doo u}\left[\frac{uI_{\nu}'(u)}{I_{\nu}(u)}\right]=\frac{\doo}{\doo u}\left[\nu+2\sum_{n\geq 1}\frac{u^2}{u^2+j_{\nu,n}^2}\right]=
4\sum_{n\geq 1}\frac{uj_{\nu,n}^2}{(u^2+j_{\nu,n}^2)^2}>0$$ for all
$u>0$ and $\nu>-1.$ We note that this proof reveals that the
function $u\mapsto uI_{\nu}'(u)/I_{\nu}(u)$ is in fact strictly
decreasing on $(-\infty,0)$ for all $\nu>-1.$ This is in the
agreement with the fact that the function $u\mapsto
uI_{\nu}'(u)/I_{\nu}(u)$ is even, as we can see in the above series
representations.

{\bf (c)} Owing to Duff \cite{duff} it is known that the function
$u\mapsto \sqrt{u}K_{\nu}(u)$ is strictly completely monotonic, and
consequently (see \cite[p. 167]{widder}) strictly log-convex on
$(0,\infty)$ for each $|\nu|\geq1/2.$ On the other hand, due to
Hartman \cite{hartman} the function $u\mapsto uI_{\nu}(u)K_{\nu}(u)$
is concave, and consequently log-concave on $(0,\infty)$ for all
$\nu>1/2.$ Since $u\mapsto 2uI_{1/2}(u)K_{1/2}(u)=1-e^{-2u}$ is
concave on $(0,\infty),$ we conclude that in fact the function
$u\mapsto uI_{\nu}(u)K_{\nu}(u)$ is concave, and hence log-concave
on $(0,\infty)$ for all $\nu\geq 1/2.$ Now, combining these results,
in view of the fact that the product of log-concave functions is
log-concave, the required result follows.

{\bf (d)} Using \eqref{eq4} and \eqref{eq3} we obtain that
\begin{align*}\frac{\doo}{\doo u}\left[\frac{u^2I_{\nu}'(u)}{I_{\nu}^2(u)}\right]&=\frac{1}{I_{\nu}(u)}
\left[u^2+\nu^2+\frac{uI_{\nu}'(u)}{I_{\nu}(u)}-2\left[\frac{uI_{\nu}'(u)}{I_{\nu}(u)}\right]^2\right]\\
&<\left[u^2+\nu^2+\sqrt{u^2+\nu^2}-2\left(u^2\frac{\nu}{\nu+1}+\nu^2\right)\right]\end{align*}
for all $u>0$ and $\nu>0.$ Observe that the last expression is
nonpositive if and only if we have
$$\left(\frac{\nu-1}{\nu+1}\right)^2u^4+\left(2\nu^2\frac{\nu-1}{\nu+1}-1\right)u^2+\nu^2(\nu^2-1)\geq0.$$
A computation shows that this is satisfied if
$$\left(2\nu^2\frac{\nu-1}{\nu+1}-1\right)^2-4\left(\frac{\nu-1}{\nu+1}\right)^2\nu^2(\nu^2-1)=-\frac{8\nu^3-9\nu^2-2\nu-1}{(\nu+1)^2}\leq0.$$
Now, since $\nu\geq \nu_0$ we have $8\nu^3-9\nu^2-2\nu-1\geq0$ and
thus the proof of part {\bf (d)} is complete.

It should be mentioned here that part {\bf (a)} of this theorem for
$\nu\geq \nu_0$ actually is an immediate consequence of this part.
More precisely, the proof of part {\bf (a)} of this theorem can be
simplified significantly as follows: in view of part {\bf (d)} of this
theorem, the function
$$u\mapsto \frac{uI_{\nu}'(u)}{I_{\nu}^2(u)}=\frac{1}{u}\cdot
\frac{u^2I_{\nu}'(u)}{I_{\nu}^2(u)}$$ is strictly decreasing as a
product of two positive and strictly decreasing functions.

Now, let us focus on the chain of inequalities \eqref{chain1}. To
prove this we use Corollary 2.5 from \cite{anderson}. More
precisely, the first inequality in \eqref{chain1} follows from part
{\bf (d)} of this theorem, while the second inequality in
\eqref{chain1} is an immediate consequence of the fact that
$I_{\nu}$ is a strictly increasing function on $(0,\infty)$ for all
$\nu>-1.$ The third inequality in \eqref{chain1} means actually the
strict geometrical convexity of $I_{\nu}$ and is equivalent to part
{\bf (b)} of this theorem; the fourth inequality is equivalent to
part {\bf (c)} of this theorem.

Finally, observe that part {\bf (a)} of this
theorem is equivalent to the inequality
$$\frac{2I_{\nu}(u_1)I_{\nu}(u_2)}{I_{\nu}(u_1)+I_{\nu}(u_2)}\leq
I_{\nu}\left(\sqrt{u_1u_2}\right),$$ which holds for all $u_1,u_2>0$
and $\nu\geq 1. $ Moreover, in this inequality equality holds if and
only if $u_1=u_2.$
\end{proof}

The following result is a companion of Theorem \ref{th1} for
modified Bessel functions of the second kind. We note that part {\bf
(b)} of the following theorem is well-known (see for example
\cite{giordano,sun,temme}), and part {\bf (c)} was proved by Baricz
\cite{baricz2}. For part {\bf (b)} we give here a different proof,
while for part {\bf (c)} we recall the proof from \cite{baricz2} and
we present a simple alternative proof.

\begin{theorem}\label{th2}
The following assertions are true:
\begin{enumerate}
\item[\bf (a)] $u\mapsto K_{\nu}'(u)/K_{\nu}^2(u)$ is
strictly decreasing on $(0,\infty)$ for all $|\nu|\geq 1;$
\item[\bf (b)] $u\mapsto K_{\nu}'(u)/K_{\nu}(u)$ is strictly increasing on
$(0,\infty)$ for all $\nu\in\mathbb{R};$
\item[\bf (c)] $u\mapsto uK_{\nu}'(u)/K_{\nu}(u)$ is
strictly decreasing on $(0,\infty)$ for all $\nu\in\mathbb{R};$
\item[\bf (d)] $u\mapsto uK_{\nu}'(u)$ is strictly
increasing on $(0,\infty)$ for all $\nu\in\mathbb{R};$
\item[\bf (e)] $u\mapsto u^2K_{\nu}'(u)$ is strictly
increasing on $(0,\infty)$ for all $|\nu|\geq 5/4;$
\item[\bf (f)] $u\mapsto u^2K_{\nu}'(u)$ is strictly
increasing on $(2,\infty)$ for all $\nu\in\mathbb{R}.$
\end{enumerate}
In particular, for all $u_1,u_2>0$ and $|\nu|\geq 1$ the following
chain of inequalities holds
\begin{equation}\label{chain2}\frac{2K_{\nu}(u_1)K_{\nu}(u_2)}{K_{\nu}(u_1)+K_{\nu}(u_2)}\leq
K_{\nu}\left(\frac{u_1+u_2}{2}\right)\leq
\sqrt{K_{\nu}(u_1)K_{\nu}(u_2)} \leq
K_{\nu}\left(\sqrt{u_1u_2}\right)\leq
\frac{K_{\nu}(u_1)+K_{\nu}(u_2)}{2}.\end{equation} Moreover, the
second, third and fourth inequalities hold true for all
$\nu\in\mathbb{R}.$ In addition, for $|\nu|\geq 5/4$ and $u_1,u_2>0$
the fourth inequality can be improved as
\begin{equation}\label{chain20}K_{\nu}\left(\frac{2u_1u_2}{u_1+u_2}\right)\leq\frac{K_{\nu}(u_1)+K_{\nu}(u_2)}{2}.\end{equation}
This inequality holds true for all $u_1,u_2>2$ and
$\nu\in\mathbb{R}.$ In each of the above inequalities equality hold
if and only if $u_1=u_2.$
\end{theorem}
\begin{proof}[\bf Proof]
First we prove the monotonicity properties for modified Bessel
functions of the second kind.

{\bf (a)} Recall that the modified Bessel function of the second
kind $K_{\nu}$ is a particular solution of the second-order
differential equation \eqref{diffeq}, and this in turn implies that
\begin{equation}\label{eq2}K_{\nu}''(u)=(1+\nu^2/u^2)K_{\nu}(u)-(1/u)K_{\nu}'(u).\end{equation}
Consequently, by using two times the right-hand side of \eqref{eq1},
for all $u>0$ and $\nu\geq 1$ we have
\begin{align*}
\frac{\doo}{\doo u}\left[\frac{K_{\nu}'(u)}{K_{\nu}^2(u)}\right]&=
\left[\frac{1}{u^2K_{\nu}(u)}\right]\left[u^2+\nu^2-\frac{uK_{\nu}'(u)}{K_{\nu}(u)}-2\left[\frac{uK_{\nu}'(u)}{K_{\nu}(u)}\right]^2\right]\\
&<-\left[\frac{1}{u^2K_{\nu}(u)}\right]\left[\frac{uK_{\nu}'(u)}{K_{\nu}(u)}\right]\left[\frac{uK_{\nu}'(u)}{K_{\nu}(u)}+1\right]\leq0.
\end{align*}
On the other hand the function $\nu\mapsto K_{\nu}(u)$ is even, and
thus from the above result we obtain that indeed the function
$u\mapsto K_{\nu}'(u)/K_{\nu}^2(u)$ is strictly decreasing on
$(0,\infty)$ for all $|\nu|\geq 1.$

{\bf (b)} The fact that $u\mapsto K_{\nu}(u)$ is log-convex can be
verified (see for example \cite{giordano,sun}) by using the
H\"older-Rogers inequality and the familiar integral representation
\eqref{integr}, which holds for each $u>0$ and $\nu\in\mathbb{R}.$
However, in view of \eqref{integr}, for all $n\in\{0,1,2,\dots\},$
$u>0$ and $\nu\in\mathbb{R},$ we easily have
$$(-1)^nK_{\nu}^{(n)}(u)=\int_0^{\infty}(\cosh t)^ne^{-u\cosh t}\cosh(\nu t)\doo t>0,$$
i.e. the function $u\mapsto K_{\nu}(u)$ is strictly completely
monotonic. Now, since each strictly completely monotonic function is
strictly log-convex, we obtain that $u\mapsto
K_{\nu}'(u)/K_{\nu}(u)$ is strictly increasing on $(0,\infty)$ for
all $\nu\in\mathbb{R}.$

{\bf (c)} Consider the Tur\'anian
$$\Delta_{\nu}(u)=K_{\nu}^2(u)-K_{\nu-1}(u)K_{\nu+1}(u).$$ Using the
recurrence relations $$K_{\nu-1}(u)=-(\nu/u)K_{\nu}(u)-K_{\nu}'(u)$$
and $$K_{\nu+1}(u)=(\nu/u)K_{\nu}(u)-K_{\nu}'(u)$$ we have
$$\Delta_{\nu}(u)=(1+\nu^2/u^2)K_{\nu}^2(u)-\left[K_{\nu}'(u)\right]^2.$$
Combining this with \eqref{eq2}, we obtain \cite{baricz2}
$$\Delta_{\nu}(u)=\frac{1}{u}K_{\nu}^2(u)\left[\frac{uK_{\nu}'(u)}{K_{\nu}(u)}\right]'.$$
But, the function $\nu\mapsto K_{\nu}(u)$ is strictly log-convex on
$\mathbb{R}$ for each fixed $u>0$ (see \cite{bariczstudia}), which
implies that for all $\nu\in\mathbb{R}$ and $u>0$ the Tur\'an-type
inequality $\Delta_{\nu}(u)<0$ holds. This shows that the function
$u\mapsto uK_{\nu}'(u)/K_{\nu}(u)$ is strictly decreasing on
$(0,\infty)$ for all $\nu\in\mathbb{R}.$ Another proof for this part
can be obtained as follows. First observe that the function
$u\mapsto uK_{\nu}'(u)/K_{\nu}(u)$ satisfies the differential
equation $uv'(u)=u^2+\nu^2-v^2(u).$ On the other hand, it is
well-known that $K_{\nu}$ is strictly decreasing on $(0,\infty)$ for
all $\nu\in\mathbb{R}.$ Thus, by using the right-hand side of
\eqref{eq1} we conclude that $u\mapsto uK_{\nu}'(u)/K_{\nu}(u)$ is
strictly decreasing too on $(0,\infty)$ for all $\nu\in\mathbb{R}.$
It is important to add here that in fact the right-hand side of
\eqref{eq1} and the Tur\'an-type inequality $\Delta_{\nu}(u)>0$ are
equivalent (see \cite{baricz1,baricz2}).

{\bf (d)} By using again the fact that $K_{\nu}$ is a particular
solution of the modified Bessel differential equation, i.e. the
relation \eqref{eq2}, we easily have for all $u>0$ and
$\nu\in\mathbb{R}$
$$\left[uK_{\nu}'(u)\right]'=K_{\nu}'(u)+uK_{\nu}''(u)=u(1+\nu^2/u^2)K_{\nu}(u)>0.$$

{\bf (e)} Using \eqref{eq2} and the left-hand side of \eqref{eq1}, we
obtain
\begin{align*}\frac{\left[u^2K_{\nu}'(u)\right]'}{K_{\nu}(u)}&=2\frac{uK_{\nu}'(u)}{K_{\nu}(u)}+\frac{u^2K_{\nu}''(u)}{K_{\nu}(u)}=
\left[\frac{uK_{\nu}'(u)}{K_{\nu}(u)}+u^2+\nu^2\right]>u^2+\nu^2-\sqrt{u^2\nu/(\nu-1)+\nu^2}\end{align*}
for all $u>0$ and $\nu>1.$ The right-hand side of the above
inequality is positive if and only if the expression
$$Q_\nu(u)=u^4+[2\nu^2-\nu/(\nu-1)]u^2+\nu^2(\nu^2-1)$$ is positive.
It is easy to see that the discriminant of the equation
$Q_{\nu}(\sqrt{u})=0$ is $(5-4\nu)\nu^2/(\nu-1)^2$ and this is
negative if and only if $\nu\geq5/4.$ Finally, since the function
$\nu\mapsto K_{\nu}(u)$ is even, the proof is complete.

{\bf (f)} In view of \eqref{integr} we obtain that
$$u^2K_{\nu}'(u)=-u^2\int_0^{\infty}e^{-u\cosh t}(\cosh t)(\cosh(\nu t))\doo t$$
and thus
$$\left[u^2K_{\nu}'(u)\right]'=u\int_0^{\infty}(u\cosh t-2)e^{-u\cosh t}(\cosh t)(\cosh(\nu t))\doo t>0$$
for all $u>2$ and $\nu\in\mathbb{R}.$

Now, let us focus on the inequalities \eqref{chain2} and
\eqref{chain20}. As in the proof of the chain of inequalities
\eqref{chain1}, we use Corollary 2.5 from \cite{anderson}. The first
inequality in \eqref{chain2} follows from part {\bf (a)}, the second
inequality is just the strict log-convexity of $K_{\nu}$ proved in
part {\bf (b)}, while the third inequality is equivalent to the
geometrical concavity of $K_{\nu}$ proved in part {\bf (c)}. The
fourth inequality is equivalent to part {\bf (d)} of this theorem,
while the inequality \eqref{chain20} is equivalent to part {\bf
(e)}.
\end{proof}



\section{\bf Convexity of modified Bessel functions with respect to power means}

In this section we are going to complement and extend the results of
the above section. To this aim we study the convexity of modified
Bessel functions of the first and second kinds with respect to
H\"older means. For reader's convenience we recall here first some
basics.

Let $\varphi:[a,b]\subseteq\mathbb{R}\rightarrow\mathbb{R}$ be a
strictly monotonic continuous function. The function
$M_{\varphi}:[a,b]^2\rightarrow [a,b],$ defined by
$$M_{\varphi}(u_1,u_2)=\varphi^{-1}\left(\frac{\varphi(u_1)+\varphi(u_2)}{2}\right)$$
is called the quasi-arithmetic mean (or Kolmogorov mean) associated
to $\varphi,$ while the function $\varphi$ is called a generating
function (or a Kolmogorov-Nagumo function) of the quasi-arithmetic
mean $M_{\varphi}.$ A function
$f:[a,b]\subseteq\mathbb{R}\rightarrow \mathbb{R}$ is said to be
convex with respect to the mean $M_{\varphi}$ (or
$M_{\varphi}-$convex) if for all $u_1,u_2\in [a,b]$ and all
$\lambda\in[0,1]$ the inequality
$$f(M_{\varphi}^{(\lambda)}(u_1,u_2))\leq
M_{\varphi}^{(\lambda)}(f(u_1),f(u_2))$$ holds, where
$M_{\varphi}^{(\lambda)}(u_1,u_2)=\varphi^{-1}(\lambda\varphi(u_1)+(1-\lambda)\varphi(u_2))$
is the weighted version of $M_{\varphi}.$ It can be proved easily
(see for example \cite{borwein}) that $f$ is convex with respect to
$M_{\varphi}$ if and only if $\varphi\circ f\circ\varphi^{-1}$ is
convex in the usual sense on $\varphi([a,b]).$ Now, for any two
quasi-arithmetic means $M_{\varphi}$ and $M_{\psi}$ (with
Kolmogorov-Nagumo functions $\varphi$ and $\psi$ defined on
intervals $[a,b]$ and $[c,d]$), a function $f:[a,b]\to[c,d]$ is
called $(M_{\varphi},M_{\psi})-$convex if it satisfies
$$f(M_{\varphi}^{(\lambda)}(u_1,u_2))\leq M_{\psi}^{(\lambda)}(f(u_1),f(u_2))$$
for all $u_1,u_2\in[a,b]$ and $\lambda\in[0,1],$ where
$M_{\psi}^{(\lambda)}(u_1,u_2)=\psi^{-1}(\lambda\psi(u_1)+(1-\lambda)\psi(u_2)).$
If the above inequality is reversed, then we say that $f$ is
$(M_{\varphi},M_{\psi})-$concave. Due to Acz\'el \cite{acel} it is
known from a long time ago that if $\psi$ is increasing then the
function $f$ is $(M_{\varphi},M_{\psi})-$convex if and only if the
function $\psi\circ f\circ \varphi^{-1}$ is convex in the usual
sense on $\varphi([a,b]).$ This is because, if $\psi$ is increasing
and we denote with $s$ and $t$ the values $\varphi(u_1)$ and
$\varphi(u_2)$, then by definition $f$ is
$(M_{\varphi},M_{\psi})-$convex if and only if
$$\psi\left(f\left(\varphi^{-1}(\lambda s+(1-\lambda)t)\right)\right)\leq
\lambda \psi\left(f\left(\varphi^{-1}(s)\right)\right)+(1-\lambda)
\psi\left(f\left(\varphi^{-1}(t)\right)\right)$$ holds for all
$s,t\in\varphi([a,b])$ and $\lambda\in[0,1].$ See also
\cite{borwein} for more details.

Now, if $\psi$ is decreasing, then clearly the above inequality is
reversed, and this in turn implies that the function $f$ is
$(M_{\varphi},M_{\psi})-$convex if and only if the function
$\psi\circ f\circ \varphi^{-1}$ is concave in the usual sense on
$\varphi([a,b]).$ Moreover, a similar characterization of
$(M_{\varphi},M_{\psi})-$concave functions is also valid, depending
on the monotonicity of the function $\psi.$

Among the quasi-arithmetic means the H\"older means (or power means)
are of special interest. They are associated to the generating
function $\varphi_p:(0,\infty)\rightarrow\mathbb{R},$ defined by
$$\varphi_p(u)=\left\{\begin{array}{ll}u^p,& \mbox{if}\ p\neq 0\\
\ln u,& \mbox{if}\ p=0,\end{array}\right.$$ and have the following
form
$$M_{\varphi_p}^{(\lambda)}(u_1,u_2)=\left\{\begin{array}{ll}
{[(1-\lambda)u_1^p+\lambda u_2^p]^{1/p}},& \mbox{if}\ p\neq 0\\
u_1^{\lambda}u_2^{1-\lambda},& \mbox{if}\ p=0.\end{array}\right.$$

Now, let $p$ and $q$ be two arbitrary real numbers. Using the above
definitions of generalized convexities we say that a function
$f:[a,b]\subseteq(0,\infty)\to(0,\infty)$ is
$(M_{\varphi_p},M_{\varphi_q})-$convex, or simply $(p,q)-$convex, if
the inequality
\begin{equation}
f(M_{\varphi_p}^{(\lambda)}(u_1,u_2))\leq
M_{\varphi_q}^{(\lambda)}(f(u_1),f(u_2))
\end{equation}
is valid for all $p,q\in\mathbb{R},$ $u_1,u_2\in[a,b]$ and
$\lambda\in[0,1].$ If the above inequality is reversed, then we say
that the function $f$ is $(M_{\varphi_p},M_{\varphi_q})-$concave, or
simply $(p,q)-$concave. Observe that the $(1,1)-$convexity is the
usual convexity, the $(1,0)-$convexity is exactly the log-convexity,
while the $(0,0)-$convexity corresponds to the case of the
geometrical convexity. We note that motivated by the works
\cite{anderson,bariczj1} and \cite{bariczj2}, recently Baricz
\cite{bariczjipam} considered the $(p,p)-$convexity of the
zero-balanced Gaussian hypergeometric functions and general power
series. The $(p,q)-$convexity of zero-balanced Gaussian
hypergeometric functions was considered recently by Zhang et al.
\cite{zhang}.

The following result gives a characterization of differentiable
$(p,q)-$convex functions and will be applied in the sequel in the
study of the convexity of modified Bessel functions of the first and
second kinds with respect to power means. For a proof see
\cite{bariczgeo}.

\begin{Lemma}\label{lem3}
Let $p,q\in\mathbb{R}$ and let
$f:[a,b]\subseteq(0,\infty)\to(0,\infty)$ be a differentiable
function. The function $f$ is (strictly) $(p,q)-$convex
($(p,q)-$concave) if and only if $u\mapsto u^{1-p}f'(u)[f(u)]^{q-1}$
is (strictly) increasing (decreasing) on $[a,b].$
\end{Lemma}

The next result completes and extends parts {\bf (a)}, {\bf (b)} and
{\bf (d)} of Theorem \ref{th1}. Notice that if we choose in part
{\bf (b)} of Theorem \ref{th1new} the values $p=0$ and $q=-1,$ then
we reobtain part {\bf (a)} of Theorem \ref{th1}. Similarly, choosing
$p=q=0$ in part {\bf (a)} of Theorem \ref{th1new} we obtain the
strict geometrical convexity stated in part {\bf (b)} of Theorem
\ref{th1}. Finally, by taking $p=q=-1$ in part {\bf (b)} of Theorem
\ref{th1new} we obtain the monotonicity result stated in part {\bf
(d)} of Theorem \ref{th1}.

\begin{theorem}\label{th1new}
Let $p,q\in\mathbb{R}$ and let $\nu>-1.$ Then the following
assertions are true:
\begin{enumerate}
\item[\bf (a)] if $p\leq0$ and $q\geq 0,$ then $I_{\nu}$ is strictly
$(p,q)-$convex on $(0,\infty);$
\item[\bf (b)] if $p\leq0$ and $q<0,$ then $I_{\nu}$ is strictly
$(p,q)-$concave on $(0,\infty)$ provided if $\nu\geq -1/q$ and
$$4q(q-1)\nu^3-(p^2-4(q-1))\nu^2-2p^2\nu-p^2\geq0;$$
\item[\bf (c)] if $p\geq0$ and $q\leq-1,$ then $I_{\nu}$ is strictly
$(p,q)-$concave on $(0,\infty)$ provided if $\nu\geq 1;$
\item[\bf (d)] if $p\geq 0$ and $q>0,$ then $I_{\nu}$ is strictly
$(p,q)-$convex on $(0,\infty)$ provided if $\nu\geq p/q;$
\item[\bf (e)] if $p\leq1$ and $q\geq 1,$ then $I_{\nu}$ is strictly
$(p,q)-$convex on $(0,\infty).$
\end{enumerate}

\end{theorem}

\begin{proof}[\bf Proof] For convenience first we introduce the
following notation
\begin{align*}
\lambda_{p,q,\nu}(u)=\frac{\doo}{\doo
u}\left[\frac{u^{1-p}I_{\nu}'(u)}{I_{\nu}^{1-q}(u)}\right]=
\frac{I_{\nu}^q(u)}{u^{p+1}}\left[u^2+\nu^2-p\left[\frac{uI_{\nu}'(u)}{I_{\nu}(u)}\right]-(1-q)\left[\frac{uI_{\nu}'(u)}{I_{\nu}(u)}\right]^2\right].
\end{align*}
We note that in view of Lemma \ref{lem3} the $(p,q)-$convexity
($(p,q)-$concavity) of $I_{\nu}$ depends only on the sign of the
expression $\lambda_{p,q,\nu}(u).$

{\bf (a)} This follows easily from the fact that if $\nu>-1,$
$p\leq0$ and $q\geq 0,$ then $\lambda_{p,q,\nu}(u)>0$ for all $u>0.$
More precisely, from the right-hand side of \eqref{eq4} we have
$$\lambda_{p,q,\nu}(u)>\frac{I_{\nu}^q(u)}{u^{p+1}}
\left[-p\left[\frac{uI_{\nu}'(u)}{I_{\nu}(u)}\right]+q\left[\frac{uI_{\nu}'(u)}{I_{\nu}(u)}\right]^2\right]\geq0$$
for all $\nu>-1,$ $p\leq0,$ $q\geq 0$ and $u>0.$ It should be
mentioned here that this part follows actually from part {\bf (b)}
of Theorem \ref{th1}. Namely, the function $u\mapsto
u^{1-p}I_{\nu}'(u)\left[I_{\nu}(u)\right]^{q-1}$ is strictly
increasing on $(0,\infty)$ for all $p\leq0,$ $q\geq0$ and $\nu>-1$
as a product of the strictly increasing functions $u\mapsto
uI_{\nu}'(u)/I_{\nu}(u)$ and $u\mapsto u^{-p}I_{\nu}^q(u).$ Now,
since for $p=q=0$ this part reduces to part {\bf (b)} of Theorem
\ref{th1}, the above remark reveals that in fact part {\bf (b)} of
Theorem \ref{th1} and part {\bf (a)} of Theorem \ref{th1new} are
equivalent.

{\bf (b)} First assume that $p<0$ and $q<0.$ Then by using
\eqref{eq4} we obtain that
$$\lambda_{p,q,\nu}(u)<\frac{I_{\nu}^q(u)}{u^{p+1}}\left[u^2+\nu^2-p\sqrt{u^2+\nu^2}-(1-q)\left(\frac{\nu}{\nu+1}u^2+\nu^2\right)\right]$$
and this is nonpositive if
$$p^2(u^2+\nu^2)\leq\left(q\nu^2+\frac{q\nu+1}{\nu+1}u^2\right)^2,\ \ \ \mbox{i.e.}\ \ 0\leq Q_{\nu}(u^2),$$
where $Q_{\nu}(u)=au^2+bu+c$ with $\nu\geq -1/q,$
$$a=\left(\frac{q\nu+1}{\nu+1}\right)^2,\ b=2q\nu^2\frac{q\nu+1}{\nu+1}-p^2,\ c=\nu^2(q^2\nu^2-p^2).$$
This gives a necessary condition to be $b^2-4ac\leq0.$ A computation
shows that the condition $b^2-4ac\leq0$ is equivalent to the
inequality
$$4q(q-1)\nu^3-(p^2-4(q-1))\nu^2-2p^2\nu-p^2\geq0.$$

Now, assume that $p=0$ and $q<0.$ Then from the left-hand side of
\eqref{eq4} we have
$$\lambda_{0,q,\nu}(u)=\frac{I_{\nu}^q(u)}{u}\left[u^2+\nu^2-(1-q)\left[\frac{uI_{\nu}'(u)}{I_{\nu}(u)}\right]^2\right]<
\frac{I_{\nu}^q(u)}{u}\left[\left(\frac{q\nu+1}{\nu+1}\right)u^2+q\nu^2\right]<0$$
for all $\nu\geq -1/q,$ $q<0$ and $u>0,$ as we requested.

{\bf (c)} This follows directly from part {\bf (a)} of Theorem
\ref{th1}. More precisely, it is easy to see that the function
$u\mapsto u^{1-p}I_{\nu}'(u)I_{\nu}^{q-1}(u)$ is strictly decreasing
on $(0,\infty)$ for all $\nu\geq 1$ as a product of the strictly
decreasing function $u\mapsto uI_{\nu}'(u)/I_{\nu}^2(u)$ and the
decreasing function $u\mapsto u^{-p}I_{\nu}^{q+1}(u).$ Since part
{\bf (c)} of Theorem \ref{th1new} reduces to part {\bf (a)} of
Theorem \ref{th1} when $p=0$ and $q=-1$, the above proof reveals
that in fact part {\bf (c)} of Theorem \ref{th1new} is equivalent to
part {\bf (a)} of Theorem \ref{th1}.

{\bf (d)} Recall that part {\bf (b)} of Theorem \ref{th1} states
that $I_{\nu}$ is strictly geometrically convex on $(0,\infty)$ for
all $\nu>-1,$ i.e. the function $u\mapsto uI_{\nu}'(u)/I_{\nu}(u)$
is strictly increasing on $(0,\infty)$ for all $\nu>-1.$ To prove
that $I_{\nu}$ is strictly $(p,q)-$convex on $(0,\infty)$ for all
$p\geq 0,$ $q>0$ and $\nu\geq p/q$ in what follows we show that the
function $u\mapsto u^{1-p}I_{\nu}'(u)I_{\nu}^{q-1}(u)$ is strictly
increasing as a product of the strictly increasing functions
$u\mapsto uI_{\nu}'(u)/I_{\nu}(u)$ and $u\mapsto
u^{-p}I_{\nu}^q(u).$ On the other hand, observe that since $u\mapsto
uI_{\nu}'(u)/I_{\nu}(u)$ is strictly increasing on $(0,\infty)$, we
obtain that
$$uI_{\nu}'(u)/I_{\nu}(u)>\nu$$ for all $\nu>-1$ and $u>0$
(actually for $\nu>0$ this inequality follows directly from the left-hand side of \eqref{eq4}). Here we
used that if $u$ tends to zero then $uI_{\nu}'(u)/I_{\nu}(u)$ tends
to $\nu,$ which can be verified from \eqref{eq4} or from
$$\frac{uI_{\nu}'(u)}{I_{\nu}(u)}=\nu+2\sum_{n\geq
1}\frac{u^2}{u^2+j_{\nu,n}^2}.$$ The above inequality implies that
$$\frac{\doo}{\doo u}\left[\frac{I_{\nu}^q(u)}{u^p}\right]=\frac{I_{\nu}^q(u)}{u^{p+1}}\left[-p+q\frac{uI_{\nu}'(u)}{I_{\nu}(u)}\right]
>\frac{I_{\nu}^q(u)}{u^{p+1}}(-p+q\nu)\geq0,$$
and with this the proof of this part is complete.

{\bf (e)} This follows from the fact that $I_{\nu}$ is strictly
increasing and convex on $(0,\infty)$ for all $\nu>-1.$ Namely, the
function $u\mapsto u^{1-p}I_{\nu}'(u)I_{\nu}^{q-1}(u)$ is strictly
increasing as a product of the strictly increasing function
$u\mapsto I_{\nu}'(u)$ and the increasing functions $u\mapsto
u^{1-p}$ and $u\mapsto I_{\nu}^{q-1}(u).$
\end{proof}

Now, we are going to present the analogous result of Theorem
\ref{th1new} for modified Bessel functions of the second kind. We
note that part {\bf (c)} of Theorem \ref{th2new} (when $p=1$ and
$q=-1$) reduces to part {\bf (a)} of Theorem \ref{th2}, part {\bf
(e)} of Theorem \ref{th2new} (when $p=1$ and $q=0$) becomes part
{\bf (b)} of Theorem \ref{th2}, part {\bf (b)} of Theorem
\ref{th2new} (when $p=q=0$) reduces to part {\bf (c)} of Theorem
\ref{th2}, and part {\bf (d)} of Theorem \ref{th2new} (when $p=0$
and $q=1$) becomes part {\bf (d)} of Theorem \ref{th2}. Finally,
observe that if we choose $p=-1$ and $q=1$ in part {\bf (a)} of
Theorem \ref{th2new}, then we obtain part {\bf (e)} of Theorem
\ref{th2}.

\begin{theorem}\label{th2new}
Let $p,q\in\mathbb{R}$ and let $\nu\in\mathbb{R}.$ Then the
following assertions are true:
\begin{enumerate}
\item[\bf (a)] if $p\leq0$ and $q\geq1,$ then $K_{\nu}$ is strictly
$(p,q)-$convex on $(0,\infty)$ provided if $\nu>1$ and
$$4(1-q)p^2\nu^2+4(q-2)p^2\nu+p^2(p^2+4)\leq0;$$
\item[\bf (b)] if $p\leq0$ and $q\leq 0,$ then $K_{\nu}$ is strictly
$(p,q)-$concave on $(0,\infty);$
\item[\bf (c)] if $p\geq0$ and $q<0,$ then $K_{\nu}$ is strictly
$(p,q)-$concave on $(0,\infty)$ provided if $|\nu|\geq-p/q;$
\item[\bf (d)] if $p\geq0$ and $q\geq 1,$ then $K_{\nu}$ is strictly
$(p,q)-$convex on $(0,\infty);$
\item[\bf (e)] if $p\geq1$ and $q\geq 0,$ then $K_{\nu}$ is strictly
$(p,q)-$convex on $(0,\infty).$
\end{enumerate}

\end{theorem}

\begin{proof}[\bf Proof] For convenience first we introduce the
following notation
\begin{align*}
\mu_{p,q,\nu}(u)=\frac{\doo}{\doo
u}\left[\frac{u^{1-p}K_{\nu}'(u)}{K_{\nu}^{1-q}(u)}\right]=
\frac{K_{\nu}^q(u)}{u^{p+1}}\left[u^2+\nu^2-p\left[\frac{uK_{\nu}'(u)}{K_{\nu}(u)}\right]-(1-q)\left[\frac{uK_{\nu}'(u)}{K_{\nu}(u)}\right]^2\right].
\end{align*}
Observe that in view of Lemma \ref{lem3} the $(p,q)-$convexity
($(p,q)-$concavity) of $K_{\nu}$ depends only on the sign of the
expression $\mu_{p,q,\nu}(u).$

{\bf (a)} Notice that for all $\nu\in\mathbb{R}$ fixed when $u$
tends to zero $uK_{\nu}'(u)/K_{\nu}(u)$ tends to $-\nu.$ This can be
verified for example from the integral representation
\eqref{integr}. On the other hand, in view of part {\bf (c)} of
Theorem \ref{th2} the function $u\mapsto uK_{\nu}'(u)/K_{\nu}(u)$ is
strictly decreasing on $(0,\infty)$ for all $\nu\in\mathbb{R},$ and
this in turn implies that for all $\nu\in\mathbb{R}$ and $u>0$ the
inequality
\begin{equation}\label{eq9}uK_{\nu}'(u)/K_{\nu}(u)<-\nu\end{equation} holds. We note that actually this
follows also from the right-hand side of \eqref{eq1}. Now, by using
\eqref{eq9} and the left-hand side of \eqref{eq1} we obtain that
$$\mu_{p,q,\nu}(u)>\frac{K_{\nu}^q(u)}{u^{p+1}}\left[u^2+\nu^2+p\sqrt{\frac{\nu}{\nu-1}u^2+\nu^2}+(q-1)\nu^2\right]$$
and the right hand side of the last inequality is nonnegative if and
only if
$$Q_{\nu}(u)=u^4+\left(2q\nu^2-\frac{\nu}{\nu-1}p^2\right)u^2+\nu^2(q^2\nu^2-p^2)\geq0.$$
Now, under assumptions the discriminant of the quadratic equation
$Q_{\nu}(\sqrt{u})=0,$ i.e.
$$\frac{\nu^2}{(\nu-1)^2}\left[4(1-q)p^2\nu^2+4(q-2)p^2\nu+p^2(p^2+4)\right]$$
is negative and with this the proof of this part is complete.

{\bf (b)} This follows from the fact that if $\nu\in\mathbb{R}$ and
$p,q\leq 0,$ then $\mu_{p,q,\nu}(u)<0$ for all $u>0.$ Namely, from
the right-hand side of \eqref{eq1} we have
$$\mu_{p,q,\nu}(u)<\frac{K_{\nu}^q(u)}{u^{p+1}}
\left[-p\left[\frac{uK_{\nu}'(u)}{K_{\nu}(u)}\right]+q\left[\frac{uK_{\nu}'(u)}{K_{\nu}(u)}\right]^2\right]\leq0$$
for all $\nu\in\mathbb{R},$ $p,q\leq 0$ and $u>0.$ Here we used that
$K_{\nu}$ is strictly decreasing on $(0,\infty)$ for all
$\nu\in\mathbb{R}.$ We note here that this part follows actually
from part {\bf (c)} of Theorem \ref{th2}. Namely, the function
$u\mapsto u^{1-p}K_{\nu}'(u)\left[K_{\nu}(u)\right]^{q-1}$ is
strictly decreasing on $(0,\infty)$ for all $p,q\leq0$ and
$\nu\in\mathbb{R}$ as a product of the strictly decreasing and
negative function $u\mapsto uK_{\nu}'(u)/K_{\nu}(u)$ and the
strictly increasing and positive function $u\mapsto
u^{-p}K_{\nu}^q(u).$ Now, since for $p=q=0$ this part reduces to
part {\bf (c)} of Theorem \ref{th2}, the above remark shows that in
fact part {\bf (c)} of Theorem \ref{th2} is equivalent to part {\bf
(b)} of Theorem \ref{th2new}.

{\bf (c)} By using \eqref{eq9} and the right-hand side of
\eqref{eq1} we have for all $u>0,$ $p\geq0,$ $q<0$ and $\nu\geq
-p/q$
\begin{align*}
\mu_{p,q,\nu}(u)&<\frac{K_{\nu}^q(u)}{u^{p+1}}
\left[-p\left[\frac{uK_{\nu}'(u)}{K_{\nu}(u)}\right]+q\left[\frac{uK_{\nu}'(u)}{K_{\nu}(u)}\right]^2\right]\\
&=-\frac{K_{\nu}^q(u)}{u^{p+1}}\left[\frac{uK_{\nu}'(u)}{K_{\nu}(u)}\right]\left[p-q\left[\frac{uK_{\nu}'(u)}{K_{\nu}(u)}\right]\right]\\
&\leq-(p+q\nu)\frac{K_{\nu}^q(u)}{u^{p+1}}\left[\frac{uK_{\nu}'(u)}{K_{\nu}(u)}\right]
\leq0.
\end{align*}

{\bf (d)} Since $p\geq0$ and $q\geq 1,$ the function $u\mapsto
u^{-p}K_{\nu}^{q-1}(u)$ is decreasing on $(0,\infty)$ for all
$\nu\in\mathbb{R}.$ Now, by using part {\bf (d)} of Theorem
\ref{th2} we conclude that $u\mapsto
u^{1-p}K_{\nu}'(u)\left[K_{\nu}(u)\right]^{q-1}$ is strictly
increasing as a product of the strictly increasing and negative
function $u\mapsto uK_{\nu}'(u)$ and the decreasing and positive
function $u\mapsto u^{-p}K_{\nu}^{q-1}(u).$ Observe that since for
$p=0$ and $q=1$ this part reduces to part {\bf (d)} of Theorem
\ref{th2}, in fact they are equivalent. Finally, we note that the
proof of this part can be obtained also simply from the fact that
under assumptions $\mu_{p,q,\nu}(u)>0.$

{\bf (e)} The proof of this part is very similar to the proof of
part {\bf (d)} above. Under assumptions the function $u\mapsto
u^{1-p}K_{\nu}^q(u)$ is decreasing. Consequently, by using part {\bf
(b)} of Theorem \ref{th2}, the function $u\mapsto
u^{1-p}K_{\nu}'(u)\left[K_{\nu}(u)\right]^{q-1}$ is strictly
increasing as a product of the strictly increasing and negative
function $u\mapsto K_{\nu}'(u)/K_{\nu}(u)$ and the decreasing and
positive function $u\mapsto u^{1-p}K_{\nu}^{q}(u).$ Observe that
since for $p=1$ and $q=0$ this part reduces to part {\bf (b)} of
Theorem \ref{th2}, in fact they are equivalent.
\end{proof}



\section{\bf Application to the log-concavity of the gamma-gamma distribution}

The probability density function
$f_{a,b,\alpha}:(0,\infty)\to(0,\infty)$ of the three parameter
gamma-gamma random variable is defined by (see \cite{karag})
$$f_{a,b,\alpha}(u)=\frac{2(ab)^{\frac{a+b}{2}}u^{\frac{a+b}{2}-1}}{\Gamma(a)\Gamma(b)\alpha^{{\frac{a+b}{2}}}}
K_{a-b}\left(2\sqrt{\frac{ab}{\alpha}u}\right),$$ where $a,b>0$ are
the distribution shaping parameters, $K_{\nu}$ stands for the
modified Bessel function of the second kind, and $\alpha>0$ is the
mean of the gamma-gamma random variable. The gamma-gamma
distribution is produced from the product of two independent gamma
random variables and has been widely used in a variety of
applications, for example in modeling various types of land and sea
radar clutters, in modeling the effects of the combined fading and
shadowing phenomena, encountered in the mobile communications
channels. Of particular interest is the application of the
gamma-gamma distribution in optical wireless systems, where
transmission of optical signals through the atmosphere is involved.
For more details see \cite{karag,karag2}.

Now, consider the functions
$\widetilde{f}_{a,b,\alpha}:(0,\infty)\to(0,\infty)$ and
$F_{a,b,\alpha}:(0,\infty)\to(0,1)$ defined by
$$\widetilde{f}_{a,b,\alpha}(u)=f_{a,b,\alpha}\left(\frac{\alpha u^2}{4ab}\right)=
\frac{2^{3-(a+b)}(ab)u^{a+b-2}}{\alpha\Gamma(a)\Gamma(b)}K_{a-b}(u)$$
and
$$F_{a,b,\alpha}(u)=\int_0^uf_{a,b,\alpha}(t)\doo t=\frac{1}{\Gamma(a)\Gamma(b)}\cdot G_{1,3}^{2,1}
\left[\left.\frac{ab}{\alpha}u\right|\begin{array}{c}1\\
a,b,0\end{array}\right],$$ where $G_{1,3}^{1,2}$ is a Meijer
$G-$function \cite[eq. 9.301]{gradshteyn}. Here
$\widetilde{f}_{a,b,\alpha}$ is just a transformation of the
probability density function ${f}_{a,b,\alpha},$ while
$F_{a,b,\alpha}$ is the cumulative distribution function of the
gamma-gamma distribution.

In probability theory usually the cumulative distribution functions
does not have closed-form, and thus sometimes it is quite difficult
to study their properties directly. In statistics, economics and
industrial engineering frequently appears some problems which are
related to the study of log-concavity (log-convexity) of some
univariate distributions. An interesting unified exposition of
related results on the log-concavity and log-convexity of many
distributions, including applications in economics, were
communicated by Bagnoli and Bergstrom \cite{bagnoli}. Some of their
main results were reconsidered by Andr\'as and Baricz \cite{andras}
by using the monotone form of l'Hospital's rule. Moreover, by using
the idea from \cite{andras}, recently, Baricz \cite{bariczgeo}
showed, among others, that if a probability density function is
geometrically concave then the corresponding cumulative distribution
function will be also geometrically concave. In this section we use
this result to prove that the cumulative distribution function
$F_{a,b,\alpha}$ is strictly log-concave on $(0,\infty)$ for all
$a,b,\alpha>0.$ This result may be useful in problems of information
theory and communications.

\begin{theorem}\label{th5}
Let $a,b,\alpha>0.$ Then the following assertions are true:
\begin{enumerate}
\item[\bf (a)] $u\mapsto
u\widetilde{f}'_{a,b,\alpha}(u)/\widetilde{f}_{a,b,\alpha}(u)$ is
strictly decreasing on $(0,\infty);$
\item[\bf (b)] $u\mapsto
u{f}'_{a,b,\alpha}(u)/{f}_{a,b,\alpha}(u)$ is strictly decreasing on
$(0,\infty);$
\item[\bf (c)] $u\mapsto
u{F}'_{a,b,\alpha}(u)/{F}_{a,b,\alpha}(u)$ is strictly decreasing on
$(0,\infty);$
\item[\bf (d)] $u\mapsto
{F}'_{a,b,\alpha}(u)/{F}_{a,b,\alpha}(u)$ is strictly decreasing on
$(0,\infty).$
\end{enumerate}
\end{theorem}

\begin{proof}[\bf Proof]
{\bf (a)} From part {\bf (c)} of Theorem \ref{th2} we have that the
function
$$u\mapsto \frac{u\widetilde{f}'_{a,b,\alpha}(u)}{\widetilde{f}_{a,b,\alpha}(u)}=a+b-2+\frac{uK_{a-b}'(u)}{K_{a-b}(u)}$$
is strictly decreasing on $(0,\infty)$ for all $a,b,\alpha>0.$

{\bf (b)} Observe that part {\bf (a)} of this theorem actually means
that the function $\widetilde{f}_{a,b,\alpha}$ is strictly
geometrically concave, i.e. for all $a,b,\alpha>0,$
$\lambda\in(0,1)$ and $u_1,u_2>0,$ $u_1\neq u_2$ we have
$$\widetilde{f}_{a,b,\alpha}\left(u_1^{\lambda}u_2^{1-\lambda}\right)>
\left[\widetilde{f}_{a,b,\alpha}(u_1)\right]^{\lambda}\left[\widetilde{f}_{a,b,\alpha}(u_2)\right]^{1-\lambda}.$$
Now, changing in the above inequality $u_i$ with
$2\sqrt{abu_i/\alpha},$ where $i\in\{1,2\},$ we obtain
$${f}_{a,b,\alpha}\left(u_1^{\lambda}u_2^{1-\lambda}\right)>
\left[{f}_{a,b,\alpha}(u_1)\right]^{\lambda}\left[{f}_{a,b,\alpha}(u_2)\right]^{1-\lambda}$$
for all $a,b,\alpha>0,$ $\lambda\in(0,1)$ and $u_1,u_2>0,$ $u_1\neq
u_2.$ This means that the function $f_{a,b,\alpha}$ is strictly
geometrically concave and hence the function $u\mapsto
u{f}'_{a,b,\alpha}(u)/{f}_{a,b,\alpha}(u)$ is strictly decreasing on
$(0,\infty).$

{\bf (c)} This follows from part {\bf (b)} of this theorem. Namely,
it is known (see \cite{bariczgeo}) that if the probability density
function is strictly geometrically concave, then the corresponding
cumulative distribution function is also strictly geometrically
concave.

{\bf (d)} Part {\bf (c)} of this theorem states that the cumulative
distribution function $F_{a,b,\alpha}$ is strictly geometrically
concave. Now, by using the fact that $F_{a,b,\alpha},$ as a
distribution function, is increasing, for all $a,b,\alpha>0,$
$\lambda\in(0,1)$ and $u_1,u_2>0,$ $u_1\neq u_2$ we have
$${F}_{a,b,\alpha}\left(\lambda u_1+(1-\lambda)u_2\right)>{F}_{a,b,\alpha}\left(u_1^{\lambda}u_2^{1-\lambda}\right)>
\left[{F}_{a,b,\alpha}(u_1)\right]^{\lambda}\left[{F}_{a,b,\alpha}(u_2)\right]^{1-\lambda},$$
that is, $F_{a,b,\alpha}$ is strictly log-concave on $(0,\infty).$
\end{proof}



\section{Open Problems}

In this section our aim is to complement the results from the previous
sections and to present certain open problems, which may be of
interest for further research.

Recall that Neuman \cite{eneuman} proved that the modified Bessel
function $I_{\nu}$ is strictly log-convex on $(0,\infty)$ for all
$\nu\in(-1/2,0].$ Since $I_{-1/2}(u)=\sqrt{\pi/(2u)}\cosh u,$ we
conclude that in fact $I_{\nu}$ is strictly log-convex on
$(0,\infty)$ for all $\nu\in[-1/2,0].$ Thus, for all
$\nu\in[-1/2,0]$ and $u_1,u_2>0$ the third inequality in
\eqref{chain1} can be improved as follows
$$I_{\nu}\left(\sqrt{u_1u_2}\right)\leq
I_{\nu}\left(\frac{u_1+u_2}{2}\right)\leq
\sqrt{I_{\nu}(u_1)I_{\nu}(u_2)}.$$

Moreover, this implies that the function $I_{\nu}$ is strictly
$(p,q)-$convex on $(0,\infty)$ for all $\nu\in[-1/2,0],$ $p\leq1$
and $q\geq0.$ This can be verified by writing the function $u\mapsto
u^{1-p}I_{\nu}'(u)I_{\nu}^{q-1}(u)$ as a product of the functions
$u\mapsto I_{\nu}'(u)/I_{\nu}(u)$ and $u\mapsto
u^{1-p}I_{\nu}^q(u).$

Concerning Theorem \ref{th1} we have the following open problem.

\begin{open}
What can we say about the monotonicity of the functions $u\mapsto
uI_{\nu}'(u)/I_{\nu}^2(u)$ and $u\mapsto
u^2I_{\nu}'(u)/I_{\nu}^2(u)$ for $|\nu|<1$ and $\nu\in(-1,\nu_0),$
respectively? Is it true that $u\mapsto \sqrt{u}I_{\nu}(u)$ is
strictly log-concave on $(0,\infty)$ for all $\nu\geq 0$?
\end{open}

Now, concerning Theorem \ref{th2}, \ref{th1new} and \ref{th2new} we may ask the following.

\begin{open}
What can we say about the monotonicity of $u\mapsto
K_{\nu}'(u)/K_{\nu}^2(u)$ when $|\nu|<1$?
\end{open}

\begin{open}
What can we say about the $(p,q)-$convexity (concavity) of $I_{\nu}$
when $p\geq0,$ $q\in(-1,0)$? Moreover, the conditions for $\nu$ in
parts {\bf (b)}, {\bf (c)} and {\bf (d)} of Theorem \ref{th1new} can
be relaxed?
\end{open}

\begin{open}
What can we say about the $(p,q)-$convexity (concavity) of $K_{\nu}$
when $p\leq1,$ $q\in(0,1)$? Moreover, the conditions for $\nu$ in
parts {\bf (a)} and {\bf (c)} of Theorem \ref{th2new} can be
relaxed?
\end{open}

It is well-known that the function $\nu\mapsto K_{\nu}(u)$ is strictly log-convex on $\mathbb{R}$ for all $u>0$
fixed (see \cite{bariczstudia}). On the other hand $\nu\mapsto K_{\nu}(u)$ is strictly increasing on $(0,\infty)$ for all $u>0$
fixed. Clearly these imply that the function $\nu\mapsto K_{\nu}(u)$ is strictly $(p,q)-$convex on $(0,\infty)$ for all $p\leq 1$ and $q\geq 0,$ and all fixed $u>0.$
This suggest the following.

\begin{open}
What can we say about the $(p,q)-$convexity (concavity) of the
function $\nu\mapsto K_{\nu}(u)$ on $(0,\infty)$ when $p$ and $q$
are arbitrary real numbers?
\end{open}

Similarly, the function $\nu\mapsto I_{\nu}(u)$ is strictly log-concave on $(-1,\infty)$ for all $u>0$
fixed (see \cite{bariczstudia}). On the other hand $\nu\mapsto I_{\nu}(u)$ is strictly decreasing on $(-1,\infty)$ for all $u>0$
fixed. Clearly these imply that the function $\nu\mapsto I_{\nu}(u)$ is strictly $(p,q)-$concave on
$(0,\infty)$ for all $p\geq 1$ and $q\geq 0,$ and all fixed $u>0.$ Thus, it is natural to ask the following.

\begin{open}
What can we say about the $(p,q)-$convexity (concavity) of the
function $\nu\mapsto I_{\nu}(u)$ on $(0,\infty)$ when $p$ and $q$
are arbitrary real numbers? And what about the $(p,q)-$convexity
(concavity) of $\nu\mapsto I_{\nu}(u)$ on $(-1,\infty)$?
\end{open}

Due to Laforgia \cite{laforgia} it is known that
$K_{\nu}'(u)/K_{\nu}(u)\leq-\nu/u-1$ for all $u>0$ and
$\nu\in(0,1/2).$ First observe that the above inequality is valid
for all $\nu\in[0,1/2].$ Since $K_0'(u)=-K_1(u)$ for $\nu=0$ the
above inequality is equivalent to $K_1(u)>K_0(u),$ which is clearly
true, since the function $\nu\mapsto K_{\nu}(u)$ is strictly
increasing on $(0,\infty)$ for all $u>0$ fixed. Now, since
$K_{1/2}(u)=\sqrt{\pi/(2u)}e^{-u}$ we obtain that in Laforgia's
inequality for $\nu=1/2$ we have equality and since $\nu\mapsto
K_{\nu}(u)$ is even, we deduce that
$K_{\nu}'(u)/K_{\nu}(u)\leq-\nu/u-1$ holds true for all $u>0$ and
$|\nu|\leq 1/2,$ with equality for $\nu=1/2.$

By using this result we obtain that
\begin{align*}\frac{\left[u^2K_{\nu}'(u)\right]'}{K_{\nu}(u)}&=2\frac{uK_{\nu}'(u)}{K_{\nu}(u)}+\frac{u^2K_{\nu}''(u)}{K_{\nu}(u)}=
\left[\frac{uK_{\nu}'(u)}{K_{\nu}(u)}+u^2+\nu^2\right]\leq
u^2-u+\nu^2-\nu<0\end{align*} for all $u\in(0,1)$ and
$|\nu|\leq1/2.$ This implies that the function $u\mapsto
u^2K_{\nu}'(u)$ is strictly decreasing on $(0,1)$ for all
$|\nu|\leq1/2,$ i.e. the modified Bessel function of the second kind
$K_{\nu}$ is strictly $(-1,1)-$concave on $(0,1)$ for all
$|\nu|\leq1/2.$ This completes parts {\bf (e)} and {\bf (f)} of
Theorem \ref{th2}.

Taking into account the above discussion we may ask the following.

\begin{open}
Is it true that $u\mapsto u^2K_{\nu}'(u)$ is strictly decreasing on
$(0,2)$ for all $|\nu|\leq1/2$?
\end{open}

In reliability analysis it has been found very useful to classify
life distributions (i.e. distributions of which cumulative
distribution function satisfies $F(u)=0$ for $u\leq0$) according to
the monotonicity properties of the failure rate. By definition a
life distribution (with probability density function $f$ and
survival or reliability function $\overline{F}$) has the increasing
failure rate (IFR) property if the function $u\mapsto
f(u)/\overline{F}(u)$ is increasing on $(0,\infty).$ Since by
definition $\overline{F}(u)=1-F(u)$ for all $u>0,$ clearly we have
$\overline{F}'(u)=-f(u)$ for all $u>0.$ Thus, a life distribution is
IFR if and only if $u\mapsto -\overline{F}'(u)/\overline{F}(u)$ is
increasing on $(0,\infty),$ i.e. the reliability function
$\overline{F}$ is log-concave. It is well-known that if a
probability density function is log-concave then this implies that
the corresponding cumulative distribution function and the
complementary cumulative distribution function (or survival
function) have the same property (for more details see
\cite{andras,bagnoli,bariczgeo}). Another class of life
distributions is the NBU, which has been shown to be fundamental in
the study of replacement policies. By definition a life distribution
satisfies the new-is-better-than-used (NBU) property if $u\mapsto
\log\overline{F}(u)$ is sub-additive, i.e.
$$\overline{F}(u_1+u_2)\leq \overline{F}(u_1)\overline{F}(u_2)$$
for all $u_1,u_2>0.$ The corresponding concept of a new-is-worse-than-used (NWU) distribution is defined by reversing the above inequality. The NBU property
may be interpreted as stating that the chance $\overline{F}(u_1)$ that a new unit will survive to age $u_1$ is greater than the chance
$\overline{F}(u_1+u_2)/\overline{F}(u_2)$ that an unfailed unit of age $u_2$ will survive an additional time $u_1.$ It can be shown easily that if a life
distribution is IFR then it is NBU (see for example \cite{bariczgamma}), but the inverse implication in general does not hold. Since the most important
life distribution satisfies the NBU property it is natural to ask the following.

\begin{open}
Is it true that the gamma-gamma distribution satisfies the NBU property?
\end{open}

To answer this question it would be enough to prove that the probability density function $f_{a,b,\alpha}$ is log-concave, and for this in view of part
{\bf (b)} of Theorem \ref{th5} it is quite enough to show that $f_{a,b,\alpha}$ is increasing. Similarly, observe that for the log-concavity of
$f_{a,b,\alpha}$ we just need to show that $\widetilde{f}_{a,b,\alpha}$ is increasing and log-concave. However, by part {\bf (a)} of Theorem \ref{th5}
if $\widetilde{f}_{a,b,\alpha}$ is increasing, then it is log-concave. Thus, to prove that the gamma-gamma distribution is NBU we need to show that either
${f}_{a,b,\alpha}$ or $\widetilde{f}_{a,b,\alpha}$ is increasing.

\subsection*{Acknowledgments}

The research of \'Arp\'ad Baricz was supported by the J\'anos Bolyai
Research Scholarship of the Hungarian Academy of Sciences and by the
Romanian National Authority for Scientific Research CNCSIS-UEFISCSU,
project number PN-II-RU-PD\underline{ }388/2011.

\end{document}